\documentclass[a4paper,12pt]{article}
\usepackage[utf8]{inputenc}	
\usepackage[margin=1in]{geometry}
\usepackage{indentfirst}
\usepackage{tikz}
\usetikzlibrary{positioning}
\usetikzlibrary{patterns}

\usepackage{color}
\usepackage{amsmath}
\usepackage{amsthm}
\usepackage{mathrsfs}
\numberwithin{equation}{section}
\usepackage{enumitem}

\newtheorem{theorem}{Theorem}[section]
\newtheorem{proposition}[theorem]{Proposition}
\newtheorem{corollary}[theorem]{Corollary}
\newtheorem{lemma}[theorem]{Lemma}
\newtheorem{lemma-definition}[theorem]{Lemma-Definition}

\newtheorem{conjecture}[theorem]{Conjecture}

\theoremstyle{definition}
\newtheorem{definition}[theorem]{Definition}
\newtheorem{remark}[theorem]{Remark}

\newtheorem{example}[theorem]{Example}

\usepackage{amssymb}
\usepackage{authblk}
\usepackage{lscape}
\usepackage{bm}
\usepackage{tikz-cd}

\usepackage{graphicx}

\def\bbC{{\mathbb C}}

\def\bbR{{\mathbb R}}

\def\bbZ{{\mathbb Z}}

\def\bbQ{{\mathbb Q}}

\newcommand{\C}{{\mathbb C}}

\newcommand{\R}{{\mathbb R}}

\newcommand{\op}{\operatorname}

\AtEndDocument{\bigskip{\footnotesize
		\textsc{Arthur Melo, Instituto de Matemática Pura e Aplicada, IMPA, Rio de Janeiro, Brazil.} \par  
		\textit{E-mail}:  \texttt{ arthur.melo@impa.br}
		\\
		
		\textsc{Vinicius Ramos, Instituto de Matemática Pura e Aplicada, IMPA, Rio de Janeiro, Brazil.} \par  
		\textit{E-mail}:  \texttt{ vgbramos@impa.br}
		\\
		
		\textsc{Alejandro Vicente, University of Stavanger, Stavanger, Norway.} \par  
		\textit{E-mail}:  \texttt{ kvicente931207@gmail.com}
}}

\title{Symplectic capacities of $S^1$-invariant dynamically convex domains in $\mathbb{R}^4$}
\author{Arthur Melo, Vinicius Ramos, Alejandro Vicente}

\date{}

\begin{document}

\maketitle

\begin{abstract}
In this paper, we prove that all normalized symplectic capacities agree for dynamically convex domains in $\mathbb{C}^2$ that are invariant under any Hamiltonian $S^1$-action isotopic to the Hopf diagonal action. We also give necessary and sufficient conditions for $S^1$-invariant domains to be dynamically convex.
\end{abstract}

\section{Introduction}

A \textit{symplectic manifold} $(X,\,\omega)$ is a $2n$-dimensional manifold $X$ equipped with a nondegenerate closed 2-form $\omega$, called a \textit{symplectic form}. A very important question in symplectic geometry is whether there exists a \textit{symplectic embedding} of $(X_1,\omega_1)$ into $(X_2,\omega_2)$, i.e., a smooth embedding $\varphi:X_1\hookrightarrow X_2$ such that $\varphi^*\omega_2=\omega_1$. If such an embedding exists, we will write $ (X_1,\omega_1)\overset{s}{\hookrightarrow} (X_2,\omega_2)$. 
\begin{definition}
	A \textit{normalized symplectic capacity} is a function $c$ that assigns to each symplectic manifold $(X,\omega)$ an element in $[0,\infty]$  such that	the following conditions hold.
	\begin{itemize}
		\item (Monotonicity) If $ (X_1,\omega_1)\overset{s}{\hookrightarrow} (X_2,\omega_2)$, then
		\begin{equation}
			c(X_1,\omega_1)\leq c(X_2,\omega_2).
		\end{equation}
		\item (Conformality) For any $\lambda>0 $
		\begin{equation*}
			c(X,\lambda\omega)=\lambda\cdot c(X,\omega).
		\end{equation*}
		\item (Normalization) $$c\left(B^{2n}(r),\omega_0\right)=c\left(Z^{2n}(r),\omega_0\right)=r.$$ 
   \end{itemize}
\end{definition}
Here the symplectic ball and the symplectic cylinder are defined as $B^{2n}(r)=\{z\in \mathbb{C}^n\mid\pi(|z_1|^2+\ldots+|z_n|^2)\leq r\}$ and $Z^{2n}(r)=\{z\in \mathbb{C}^n\mid\pi|z_1|^2\leq r\}$, respectively. Moreover, $\omega_0=\sum_i dx_i\wedge dy_i$ is the standard symplectic form on $\R^{2n}\cong\C^n$.

Many symplectic capacities have been defined in several seemingly unrelated ways, such as the Hofer-Zehnder capacity in \cite{HOFERzehndercap}, the Ekeland-Hofer capacities in \cite{Ekeland1989}, the ECH capacities in \cite{Hutchings2010QuantitativeEC}. Within the class of all normalized symplectic capacities there are two special ones. The first one is referred to as the \textit{Gromov width} and is defined as:
\begin{equation*}
c_{\mathrm{Gr}}(X^{2n},\,\omega) =\sup\left\{r\mid (B^{2n}(r),\,\omega_0)\overset{s}{\hookrightarrow} (X,\,\omega)\right\}.
\end{equation*}
The second one is known as the \textit{cylindrical capacity} and is defined as
\begin{equation*}
c_Z(X^{2n},\,\omega) =\inf\left\{r\mid (X,\,\omega)\overset{s}{\hookrightarrow} (Z^{2n}(r),\,\omega_0)\right\}.
\end{equation*}

These two capacities are particularly important since it follows that for any normalized symplectic capacity $c$ and any symplectic manifold $(X,\omega)$, we have that
\begin{equation}\label{eq:cap_ineq}
c_{\mathrm{Gr}}(X,\,\omega)\leq c(X,\,\omega)\leq c_Z(X,\,\omega).
\end{equation}

In this paper, we will only consider domains $X\subset \R^{2n}$ endowed with $\omega_0$, so from now on we will omit the reference to $\omega_0$. Denote by $E(a_1,\cdots,a_{n})\subset \mathbb{C}^{n}$ the \textit{symplectic ellipsoid} given by 
\begin{equation}\label{eq: def_sympt_ellip}
\left\{(z_1,\ldots,z_{n})\in \bbC^{n}|\,\,\pi\left(\frac{|z_1|^2}{a_1}+\cdots+\frac{|z_{n}|^2}{a_{n}}\right)\leq 1 \right\}.
\end{equation}

It was first observed in \cite{hofer1994symplectic} that all normalized symplectic capacities agree in the class of symplectic ellipsoids. Therefore, a natural subsequent question would be to find a larger (or the largest) class of domains in $\mathbb{R}^{2n}$ where all normalized capacities agree. A folklore conjecture that has driven much research in the area during the last three decades is that all normalized symplectic capacities agree in the class of convex domains in $\mathbb{R}^{2n}$. This has been referred to in the literature as the \textit{strong Viterbo conjecture} due to its relation to the \textit{Viterbo conjecture}:

\begin{conjecture}[Viterbo \cite{Viterbo2000isoper}]\label{conj: Vit}
For any symplectic capacity $c$, and any convex domain $X\subset \mathbb{R}^{2n}$,
\begin{equation}\label{wvp}
c(X)^n\leq n!\op{Vol}(X).
\end{equation}
\end{conjecture}

It is not hard to show that the Gromov width satisfies \eqref{wvp}, hence the strong Viterbo conjecture implies the Viterbo conjecture.

Many partial results have been achieved for these two closely related conjectures. Of special interest is the connection found in \cite{ArtsteinAvidan2013FromSM} with the longstanding Mahler conjecture in the area of convex geometry. Namely, it was shown that the Viterbo conjecture for the  particular choice of the Hofer-Zehnder capacity and restricted to the subclass of bodies of the form $K\times K^{\circ}$, for $K\subset\mathbb{R}^n$ a centrally symmetric convex body and $K^{\circ}\subset\mathbb{R}^{n}$ its polar dual body, is equivalent to the Mahler conjecture stating that for such convex body:
$$\textup{Vol}(K)\textup{Vol}(K^{\circ})\geq \frac{4^n}{n!}.$$

However, it has been recently shown in \cite{HaimKislev2025Counterexample} that the Viterbo conjecture is not true as it was originally stated. The product of a pentagon and a rotated copy of it gives a counterexample. So, restricted to the class of convex bodies in $\mathbb{R}^{2n}$ all normalized symplectic capacities do not agree. Nevertheless, there are some subclasses where they have been shown to agree. For example, it was shown in \cite{Gutt2022Examples} that they agree in the class of monotone toric domains in $\mathbb{R}^4$ and this was later generalized to higher dimensions in \cite{cristofarogardiner2023agreementsymplecticcapacitieshigh}. Before that, Ostrover had shown that they agree in the class of convex bodies in $\mathbb{R}^{2n}$ invariant under the diagonal $S^1$-action $e^{i\theta}\cdot(z_1,\ldots,z_n)=(e^{i\theta}z_1,\ldots,e^{i\theta}z_n)$, see
\cite{Gutt2022Examples}. 
In this paper, we will give a generalization of this result, as stated in Theorem \ref{cor: result} below.

\begin{remark}
One more class of domains in $\mathbb{R}^{2n}$ where all normalized symplectic capacities are expected to agree is the class of convex domains given by the dual product of centrally symmetric convex domains, related to the Mahler conjecture mentioned above. There are some partial results in this direction for some particular examples coming from Hanner polytopes and $L^p$ balls. In \cite{vicente2025strongviterboconjecturevarious}, the third author found yet another class of dual products, with a different notion of duality, where all normalized symplectic capacities agree.
\end{remark}

\subsection{$S^1$-invariant domains}

Let  $X$  be a star-shaped domain, i.e., a compact domain in $\mathbb{C}^{n}$ with smooth connected boundary such that $0 \in \op{int}(X)$ and such that the radial vector field is transverse to $\partial X$. The induced Liouville form is given by $\lambda_0=\frac{i}{4}\sum_j(z_jd\overline{z}_j-\overline{z}_jdz_j)$. Note that the restriction of $\lambda_0$ to ${\partial X}$ is a contact form on $\partial X$. As usual, we will still denote by $\lambda_0$ the restricted 1-form on $\partial X$ for any star-shaped domain $X$.

A star-shaped domain $X$ is said to be $S^1$-invariant if there exists a free action of $S^1$ on $\partial X$ that preserves $\lambda_0$, i.e., the infinitesimal generator $V$ of the $S^1$-action satisfies $\mathcal{L}_{V}\lambda_0=0$. Any vector field $V$ of $\partial X$  can be extended radially in a homogeneous way to the whole $\bbC^{n}$. In this sense, we could also say that $X$ is $S^1$-invariant if it is invariant under the flow of some homogeneous vector field $V$ that generates a free $S^1$-action outside the origin. 
 
 In the case of $X\subset \bbC^2$, the boundary of $X$ is diffeomorphic to $S^3$. Moreover, every effective smooth circle action on $S^3$ is fully classified up to orientation-preserving equivariant diffeomorphism: they generate a Seifert fibration\footnote{ an $S^1$-bundle over a 2-dimensional orbifold.}, and correspond to one of the linear actions
\begin{equation}
	(z_1,z_2)\mapsto (e^{i\theta k}z_1,e^{i\theta l}z_2),
\end{equation}
with $\gcd(k,l)=1$ (see \cite{Epsteinflows3dim}). Moreover, if we assume that the action is free, then the only options are $k=1,l=\pm 1$. We will refer to the cases $l=1$ and $l=-1$ as the Hopf and anti-Hopf actions, respectively. These actions are conjugate, but not by an orientation-preserving map. They give rise to $S^1$-fibrations over $\mathbb{CP}^1$ with Euler numbers -1 and 1, respectively.  

\begin{definition}
	If a free $S^1$-action on $S^3$ is conjugate  to the Hopf action via an orientation-preserving equivariant diffeomorphism, we call it a \textit{positive free $S^1$-action}. Otherwise, we call it a \textit{negative free $S^1$-action}. Similarly, as the boundary of star-shaped domain $X$ is diffeomorphic to $S^3$, we say that a free $S^1$-action on $\partial X$ is positive (negative) if it is conjugate,  via an orientation-preserving equivariant diffeomorphism, to a positive (negative) free $S^1$-action on $S^3$. We say that $X$ is positive (negative) $S^1$-invariant if its respective free $S^1$-action on $\partial X$ is positive (negative).  
\end{definition}
  We now recall the definition of dynamically convex contact forms in $S^3$.

\begin{definition}\label{def: dynconv}
	Let $(S^3,\lambda)$ be a contact 3-sphere such that $\ker \lambda =\xi_{std}$. The contact form $\lambda$ is called \textit{dynamically convex} if the Conley--Zehnder index of every Reeb orbit is at least 3, which is equivalent to say that the rotation number of every Reeb orbit is strictly greater than 1. Here, we are considering the lower semi-continuous extension of the  Conley--Zehnder index.
\end{definition}

We are now ready to precisely state the main result of this paper.
\begin{theorem}\label{cor: result}
	Let $X\subset \bbR^4$ be a star-shaped positive $S^1$-invariant dynamically convex domain. Then 
	\begin{equation*}
		c_{\mathrm{Gr}}(X)= c_Z(X).
	\end{equation*}
\end{theorem}
 This result generalizes the results in \cite{Gutt2022Examples}, where it was proved that all normalized symplectic capacities agree for dynamically convex toric domains and convex domains invariant under the Hopf flow. It also partially generalizes the result in \cite{vialaret2024sharpsystolicinequalitiesinvariant}, which proved the Viterbo conjecture for convex domains in $\mathbb{R}^4$ invariant under free, symplectic $S^1$-actions.
 
 The proof of Theorem \ref{cor: result} involves finding a special orbit that is unknotted and has self-linking number $-1$ (see §\ref{sec:hopf_orbits} for the definition of self-linking number). We call any Reeb orbit on a contact 3-manifold satisfying these properties a \textit{Hopf orbit}. 
 
 The concept  of Hopf orbits was first introduced in \cite{tmna/1479265300}, where it was shown that, for any star-shaped domain $X\subset \mathbb{R}^4$, its boundary contains a Hopf orbit. Hence,  we can define the non-empty set $\mathcal{P}_{\mathrm{Hopf}}(\partial X)$ of Hopf orbits and the quantity
 $$\mathcal{A}_{\mathrm{Hopf}}(X)=\inf\left\{\mathcal{A}(\gamma)\mid\gamma\in \mathcal{P}_{\mathrm{Hopf}}(\partial X)\right\}.$$

 We can now state the following intermediate result, which will be proved in §\ref{sec:hopf_orbits}.
 \begin{theorem}\label{thm:main}
 		Let $X\subset \bbR^4$ be a star-shaped positive $S^1$-invariant domain. Then 
 		\begin{equation}\label{eq:cor_ineq}
 			\mathcal{A}_{\mathrm{Hopf}}(X)\leq c_{\mathrm{Gr}}(X)
 		\end{equation}
 \end{theorem}
 Using this result, we can now prove Theorem \ref{cor: result}.
\begin{proof}[Proof of Theorem \ref{cor: result}]
It was shown in \cite{edtmair2023disklikesurfacessectionsymplectic} that for any 
dynamically convex domain $X\subset \mathbb{C}^2$, there exists a symplectic embedding $X\hookrightarrow Z^4(a)$, if and only if, $a\geq \mathcal{A}_{\mathrm{Hopf}}(X)$. In particular,
\begin{equation}\label{eq: edtmair}
	c_{Z}(X)\le \mathcal{A}_{\mathrm{Hopf}}(X).
\end{equation}
 If $X$ is positive $S^1$-invariant, we get the inequality \eqref{eq:cor_ineq}. We conclude the proof by combining \eqref{eq:cap_ineq}, \eqref{eq:cor_ineq} and \eqref{eq: edtmair}.
\end{proof}

On the other hand, the same technique does not seem to apply to the case of dynamically convex negative $S^1$-invariant domains. This happens because there is no clear relation between the Hopf orbits and ball embeddings on those domains; see §\ref{sec:hopf_orbits} for a full discussion and partial results.

\subsection{Characterizing dynamically convex invariant domains}
The second part of this paper is devoted to characterizing dynamically convex domains in $\mathbb{R}^4$ invariant under the Hopf flow.

We recall that there is a strict contactomorphism between $(\partial X,\lambda_0)$ and $(S^{3},F\lambda_0)$, for a smooth function $F:S^{3}\rightarrow \mathbb{R}_{>0}$. In fact, given a star-shaped domain $X$, $F$ is defined by
$F(\mathbf{w})=\sup\{r^2\mid r\mathbf{w}\in X\}$. Conversely, for any function $F:S^{3}\rightarrow \mathbb{R}_{>0}$ we can define the domain
\begin{equation*}
	\mathbb{X}_F=\left\{\mathbf{z}\in \mathbb{C}^{2}\backslash\{0\}\mid |\mathbf{z}|^2\leq F\left(\frac{\mathbf{z}}{|\mathbf{z}|}\right)\right\}\cup \{0\},
\end{equation*}
which is a star-shaped domain whose boundary is strictly contactomorphic to $(S^{3},F \lambda_0)$.

Consider the Hopf flow on $S^3$ and the associated Hopf fibration $p:S^{3}\xrightarrow[]{}\bbC P^1$, where
we take a smooth function $f:\bbC P^1\xrightarrow[]{}\bbR$ and define the contact form $\lambda_f=e^{p^*f}\lambda_0$ on $S^{3}$. If $F=e^{p^*f}$, then it follows from the discussion above that $(\partial \mathbb{X}_F,\lambda_0)$ is strictly contactomorphic to $(S^{3},\lambda_f)$.

\begin{example}
	For the symplectic ellipsoid $E(a,b)\subset \mathbb{C}^{2}$ defined in \eqref{eq: def_sympt_ellip}, a straightforward computation shows that $E(a,b)=\mathbb{X}_F$ for $F=e^{p^*f}$, where $f:\mathbb{CP}^1\rightarrow \mathbb{R}$ is given by 
	$$f([z_1:z_2])=\ln\left( \dfrac{ab(|z_1|^2+|z_{2}|^2)}{b\pi|z_1|^2+a\pi|z_2|^2}\right).$$
\end{example}

This class of domains was briefly studied in \cite{hutchings2024zetafunctionsdynamicallytame}, where it was shown that, for cases where the function $f$ is Morse--Bott with more than two isolated critical points, the domain $\mathbb{X}_F$ associated to it is not symplectomorphic to a toric domain.

We can also understand $S^1$-invariant domains $\mathbb{X}_F$ from the point of view of integrable systems. Indeed, it is not hard to see that $\mathbb{X}_F$ is invariant under the Hamiltonian flow of the Poisson commuting Hamiltonians $H:\mathbb{X}_F\subset \mathbb{C}^2\to \mathbb{R}$ and $G:\mathbb{X}_F\subset \mathbb{C}^2\to \mathbb{R}$ given by:
$$H(\mathbf{z})=|\mathbf{z}|^2, \quad G(\mathbf z)=
\begin{cases}
|\mathbf z|^2\,F\!\left(\dfrac{\mathbf z}{|\mathbf z|}\right),
& \mathbf z \neq 0, \\[0.8em]
0, & \mathbf z = 0.
\end{cases}.$$

Notice that $\partial \mathbb{X}_F=(G/H-H)^{-1}(0)$. Moreover, both $H$ and $G$ are homogeneous, therefore satisfying the hypothesis in \cite{Colombo_2025}. Therefore, toric coordinates actually give a local strict contactomorphism between $\partial \mathbb{X}_F$ and toric contact manifolds. In fact, there exists a set of toric coordinates that describe (up to a well-understood zero-measure set) any $S^1$-invariant domain $\mathbb{X}_F$ with a collection of curves $\{\alpha_i\}$ in $\bbR^2_{\geq 0}$. In general, toric coordinates are not canonical and it is possible to find another set of curves by applying a transformation of $SL(2,\bbZ)$. However, we find a set of toric coordinates which is canonical up to reflection on the $y=x$ axis for this case.
 
Before expanding on this, we first recall the definition of toric domains. Let $\Omega$ be a domain in $\bbR^2_{\geq 0}$ and $\mu:\bbC^2\to \bbR^2$ the moment map for the standard toric action on $\mathbb{C}^2$ given by $\mu(z_1,z_2)=\pi(|z_1|^2,|z_2|^2)$. We then define the \textit{toric domain}
 $$X_{\Omega}=\{z \in \bbC^2\,|\,\mu(z_1,z_2)\in \Omega\}.$$
We define $\partial_+\Omega:=\partial\Omega\cap \mathbb{R}_{>0}^2$ . We can describe the boundary of $\partial X_\Omega $ fully by the curve $\overline{\partial_+\Omega}$. Moreover, we can parametrize the curve by some map $\alpha:[0,1]\mapsto \overline{\partial_+\Omega}\subset  \bbR^2_{\geq 0}$ where $\alpha(t)\in \partial_+\Omega$ for $t\in(0,1)$.
 
 We can also consider curves $\alpha:[0,1)\mapsto \bbR^2_{\geq 0}$ such that $\alpha(0)$ is the only intersection with either the axis $x=0$ or $y=0$ and  $\alpha:(0,1)\mapsto \bbR^2_{>0}$. Notice then that $\mu^{-1}(\textup{Im}(\alpha))$ is a contact toric manifold, homeomorphic to a solid torus, in the former case, or an interval times a torus, in the latter case. From now on, we refer to a \textit{piece of curve} in $\bbR^2_{\geq 0}$ as one of these three types of curves. 
 
 \begin{definition}
 	A smooth curve $\alpha$ on $\bbR^2_{\geq 0}$ is said to be \textit{monotone} if, at any point $\alpha(t)$ on the curve, the outward normal vector to $\alpha$ has positive coordinates. Notice that this property is preserved under a reflection on the $y=x$ axis.
 \end{definition}
This definition was introduced in the context of toric domains in \cite{Gutt2022Examples}, and it was proved to be equivalent to dynamical convexity in the class of toric domains. Here we extend this result for domains invariant under the Hopf flow and prove the following result.
\begin{theorem}\label{thm:characdynconv}
	Let $f:\mathbb{CP}^1\mapsto\bbR$ be a Morse--Bott function and let $\Lambda$ be the union of the connected components of $f^{-1}(\text{crit}(f))$ containing a nondegenerate saddle point of $f$. Then there is a finite family of curves $\{\alpha_i\}_{i\in I}$ in $\bbR^2_{\geq 0}$ indexed by the connected components of $\mathbb{CP}^1\setminus\Lambda$, which are canonical up to reflection on the $y=x$ axis, such that $(S^3\setminus p^{-1}(\Lambda),\lambda_f)$ is strictly contactomorphic to $\left(\bigsqcup_{i\in I}\mu^{-1}(\mathrm{im}(\alpha_i)),\lambda_0\right)$. 
	
	Moreover, $\lambda_f$ is dynamically convex if, and only if, the associated curves $\{\alpha_i\}$ in $\bbR^2_{\geq 0}$ are all monotone. 
\end{theorem}
\noindent{\bf Structure of the paper:} In Section \ref{sec:hopf_orbits} we prove Theorem \ref{thm:main} and study the case of the anti-Hopf action. In Section \ref{sec: dyn_conv} we compute the Conley--Zehnder index of distinguished and undistinguished Reeb orbits. We also prove Theorem \ref{thm:characdynconv} characterizing Hopf flow invariant dynamically convex domains.

\noindent{\bf Acknowledgements:} A.M. would like to thank Simon Vialaret for discussions on $S^1$-invariant contact forms on $S^3$ during his time in Bochum. A. M. was supported by a CNPq doctoral scholarship grant number 141615/2023-0. V.R. was partially supported by CNPq productivity grant number 308678/2025-
7, FAPERJ Jovem Cientista do Nosso Estado grant number E-26/204.621/2024 and a
grant from the Serrapilheira Institute. This work was done during A. V.  Postdoc Fellow position at the Hebrew University of Jerusalem, supported by the ISF Grant No. 2445/20.

\section{Hopf orbits on circle invariant contact forms}\label{sec:hopf_orbits}
We start by recalling the definition of the self-linking number for transverse knots. Let $K$ be a knot in $Y$ that is unknotted and positively transverse to $\xi$ with respect to some co-orientation of $\xi$. The self-linking number $\op{sl}(K)$ is defined as follows. Take $u:\mathbb{D}\to Y$ to be a capping disk of $K$ and let $Z$ be a smooth non-vanishing section of the vector bundle $u^*\xi\rightarrow \mathbb{D}$. Take $\epsilon>0$ small and consider an exponential map in $Y$ so that we can define the knot 
$$K_{\epsilon}:t\in \mathbb{R}/\mathbb{Z}\mapsto \exp_{u(e^{2\pi i t})}(\epsilon Z(u(e^{2\pi i t})))\in Y. $$
The \textit{self-linking} number of $K$ is defined as the intersection number:
$$\op{sl}(K)=\#(K_{\epsilon}\cap u),$$
where  the count of intersections is done with signs. Observe that this notion depends only on the transverse knot and on the contact structure. In particular, a Hopf fiber, i.e., any fiber of the Hopf fibration, has self-linking number $-1$ with respect to the standard contact structure.

Let $X\subset \bbC^2$ be a star-shaped positive $S^1$-invariant domain. Since the action is positive, the associated circle bundle $\partial X\to \partial X/S^1$ has Euler number $-1$. By the Giroux classification (see \cite{girouxcirclefibers,Lutz}), as used by Vialaret in \cite{vialaret2024sharpsystolicinequalitiesinvariant}, an invariant tight contact structure with Euler number $-1$ is transverse to the circle fibers. Hence $(\partial X, \lambda_0)$ is strictly contactomorphic to $(S^3,F\lambda_0)$ where $F:S^3\mapsto \bbR_{>0}$ is invariant under the Hopf flow. Then we write $F=e^{p^*f}$ and $\lambda_f=e^{p^*f}\lambda_0$. A routine computation shows that the Reeb vector field $R_f$ associated with the contact form $\lambda_f$ is given by
\begin{equation}\label{eq: preqflow}
	R_f=\frac{1}{F}(R_0+\widetilde{X}_f),
\end{equation}
where $\widetilde{X}_f\in \ker \lambda_f$ is the horizontal lift to $S^{3}$ of the Hamiltonian vector field $X_f$ associated with $f$ on $(\mathbb{C} P^1,\omega)$, and $\omega$ is the Fubini-Study symplectic form. It is not hard to see that $\omega$ satisfies that $d\lambda_0=p^*\omega$, see \cite{hutchings2024zetafunctionsdynamicallytame}. It follows from \eqref{eq: preqflow} that for every critical point $x\in \mathbb{CP}^1$, the fiber $p^{-1}(x)$ is a closed Reeb orbit, since $X_f|_x=0$. We denote this Reeb orbit by $\gamma_x$.

\begin{proof}[Proof of Theorem \ref{thm:main}]
	Let $x$ be a critical point of the function $f:\mathbb{CP} ^1\rightarrow \mathbb{R}$. Since $\gamma_x=p^{-1}(x)$, it follows that $\gamma_x$ is a Hopf fiber and therefore it is a Hopf orbit.
	
Let $m=f(x_0)$ be the global minimum of $f$. It follows from \eqref{eq: preqflow} that $\mathcal{A}(\gamma_{x_0})=\pi e^{m}$. Now, define the constant function $g:\mathbb{C} P^{1}\rightarrow \mathbb{R}$ by $g(z)=m$ for all $z \in \mathbb{CP}^1$ and let $G=e^{p^* g}$. Then $G \leq F$ and therefore $\mathbb{X}_G \subseteq \mathbb{X}_F$. Since $\mathbb{\mathbb{X}}_G$ is the symplectic ball of capacity $\pi e^m$, it follows that
	$$\mathcal{A}_{\mathrm{Hopf}}(\mathbb{X}_F)\leq \mathcal{A}(\gamma_{x_0})= \pi e^m= c_{\mathrm{Gr}}(\mathbb{\mathbb{X}}_G) \leq c_{\mathrm{Gr}}(\mathbb{X}_F)$$
\end{proof}

It is interesting to analyze what happens in the case of the $S^1$-action being the anti-Hopf action. In general, Theorem \ref{thm:main} fails for a contact form $F\lambda_0$ which is invariant under the anti-Hopf flow. This is due to the fact that critical levels of $f$, the function induced in $\mathbb{CP}^1$ by taking the quotient by the $S^1$-action, might not even contain a Reeb orbit. However, instead of looking at critical points of $F$, we can look at critical points of $A=\frac{1}{2}(|z_1|^2 - |z_2|^2) F$.

\begin{theorem}\label{thm:main2}
	Let $F:S^3\rightarrow \mathbb{R}$ be a smooth function that is invariant under the anti-Hopf action and define the function $A=\frac{1}{2}(|z_1|^2 - |z_2|^2) F$. Then, for every critical point $z$ of $A$,
	\begin{enumerate}[label=(\roman*)]
		\item $z$ lies on a Reeb orbit $\gamma$ of $F\lambda_0$.
		\item $\gamma$ is a Hopf orbit.
	\end{enumerate}	 
\end{theorem}
We shall give the proof of Theorem \ref{thm:main2} at the end of this section.

\begin{corollary}
	Let $F:S^3\rightarrow \mathbb{R}$ be a smooth function that is invariant under the anti-Hopf flow and assume either $z=(1,0)$ or $z=(0,1)$ is a point of minimum for $F$. Then
	$$\mathcal{A}_{\mathrm{Hopf}}(\mathbb{X}_F)\leq c_{\mathrm{Gr}}(\mathbb{X}_F)$$
	
\end{corollary}
\begin{proof}
This follows directly from Theorem \ref{thm:main2}.
\end{proof}

\begin{remark}
	Note that we cannot simply apply a transformation of $U(2)$ taking the desired critical point to $z=(1,0)$ or $z=(0,1)$. The proof of Theorem \ref{thm:main2} is highly dependent on which action we are taking, and a transformation of $U(2)$ does not preserve the action in general.
\end{remark}

Let $F:S^3\rightarrow \mathbb{R}$ be a smooth function that is invariant under the $(z_1,z_2)\mapsto(e^{i\theta}z_1,e^{-i\theta}z_2)$ action. This action is generated by the vector field $V=\partial_{\theta_1}-\partial_{\theta_2}$ in polar coordinates. Similarly to \cite{vialaret2024sharpsystolicinequalitiesinvariant}, define $$A=i_{V}(F\lambda_0)=\frac{1}{2}(|z_1|^2-|z_2|^2)F$$

The critical points of $A$ correspond to points where $R_{F}$, the Reeb flow associated with $F\lambda_0$, is parallel to $V$, i.e., these are anti-Hopf fibers. More precisely, we have:

\begin{proposition}(\cite{vialaret2024sharpsystolicinequalitiesinvariant})\label{prop: action}
	Let $\alpha$ be a contact form on $S^3$ that is invariant under the free $S^1$-action generated by $V$. Define $A=i_V\alpha$ and its projection $\bar{A}$ to $S^2=S^3/S^1$. The differential of $A$ is $dA=-i_V d\alpha$, and the critical points of $\bar{A}$ are in bijection with the $S^1$-orbits of critical points of $A$. These $S^1$-orbits coincide as sets with Reeb orbits. The associated critical value of $A$ is proportional to the minimal period of the corresponding Reeb orbit (up to sign), depending on whether the orientations of the Reeb orbit and the $S^1$-orbit coincide.
\end{proposition}

To compute the Reeb vector field $R$, we use the splitting $TS^3=\op{Span}(R_0)\oplus \xi_{std}$, and write $R=aR_0+b W$ with $W\in \ker \lambda_0$. In particular, if $dF=0$ at some point $z$, then
 
 \begin{equation*}
 	0=\text{i}_{R}d(F\lambda_0)=Fi_{R}d\lambda_0=F i_{bW}d\lambda_0
 \end{equation*}
Since $W\in \ker \lambda_0$, we have $R=\frac{1}{F}R_0$ at this point. However, unless $F$ is also invariant under the Hopf flow, the flow $R_{F}$ along the orbit of $z$ is not necessarily a rescaling of $R_0$.

We now use this result to prove Theorem \ref{thm:main2}.
\begin{proof}[Proof of Theorem \ref{thm:main2}]
	Let $z\in S^3$ be a critical point of $A$, by Proposition \ref{prop: action}, $z$ is part of a Reeb orbit $\gamma$ which is a rescaling of an anti-Hopf fiber. Now, observe that $\gamma$ is a Hopf orbit, because every anti-Hopf fiber, which is also a Reeb orbit, is isotopic to either ${(e^{it},0)}$ or ${(0,e^{-it})}$,i.e.,  $\gamma$ is isotopic to a Hopf fiber, then is unknotted and has self-linking number -1.  
\end{proof}

\section{Dynamically convex $S^1$-invariant domains}\label{sec: dyn_conv}

In this section, as in \cite{Gutt2022Examples}, we provide a characterization of Morse--Bott functions $f:\bbC P^1\to \bbR$ such that $(S^3,\lambda_f=e^{p^*f}\lambda_0)$ is dynamically convex. Unless otherwise stated, we will consider $F=e^{p^*f}$ and $\mathbb{X}_F$ its corresponding domain. Let $x\in\mathbb{CP}^1$ be an isolated critical point of $f$ and $\gamma_x$ the corresponding Reeb orbit for $\lambda_f$. We will refer to these orbits as \textit{distinguished Reeb orbits}. Observe that if $f$ is a Morse--Bott function, any other Reeb orbits come in $S^1$-families. We call them \textit{undistinguished Reeb orbits}. We will use a global trivialization of $\xi=\ker \lambda_0$ as follows. As in \cite{Chaidez4dconvexpolytopes}, we use the quaternions to define $\tau_0:\xi \xrightarrow{}S^3\times \bbR^2$. We see $S^3\subset \R^4=\C^2=\mathbb{H}$ under the usual identification and let $\{1,i,j,k\}$ be the usual basis of $\R^4$. Then for each $z\in S^3$, it follows from a simple calculation that $R_0|_z=iz$ and that $\xi_z=\textrm{span}(jz,kz)$. The trivialization $\tau_0$ is defined by 
$$\tau_0(v):=(z,<v,jz>,<v,kz>).$$

For the calculations in this paper, we will only use this trivialization and we will set $CZ(\gamma)=CZ_{\tau_0}(\gamma)$. 

 In the first part, we are going to explicitly calculate the Conley--Zehnder index for distinguished orbits using the Robbin-Salamon index depending on the Hessian of $f$. In the second section, we calculate the holonomy of the connection $\lambda_0$ with respect to the flow of $\widetilde{X}_f$. After that, we relate this number to action angle coordinates of the integrable system generated by $f$ and $|z|^2$. We use this description in toric coordinates to calculate the rotation number of undistinguished orbits. Finally, we use the toric description to give necessary and sufficient conditions for $\lambda_f$ to be dynamically convex.

\subsection{Conley--Zehnder of distinguished orbits}

We first recall the definition of the Robbin-Salamon index. The Robbin-Salamon index is a generalization of the Conley--Zehnder index for a more general path of symplectic matrices. For
$$\Psi\in \Sigma(n):=\{\Psi:[a,b]\xrightarrow{}Sp(n)\,\,|\,\,\Psi\text{ is smooth}\},$$
we say that a number $t\in [a,b]$ is a \textit{crossing} if $\det(\Psi(t)-\mathbb{I})=0$. For a crossing $t$ we define $E_t:=\ker(\Psi(t)-\mathbb{I})$ and the \textit{crossing form} being a quadratic form on $E_t$ given by
$$\Gamma(\Psi,t)(v):=d\lambda(v,\dot{\Psi}v), \,\,\,\,\, \text{for}\,\, v \in E_t.$$ 
A crossing $t$ is said to be \textit{regular} if the crossing form associated to it, is a non-singular quadratic form, i.e. it is non-zero for every vector on $E_t$. Observe that regular crossings are isolated, and any path can be modified by a homotopy with fixed endpoints to a path that has only regular crossings, see \cite{ROBBIN1993}.
The \textit{Robbin-Salamon} index of a path $\Psi:[a,b]\to Sp(n)$ with only regular crossings is defined by
$$\mu_{RS}(\Psi):=\frac{1}{2}\sum_{\substack{t\in\{a,b\} \\ t\ \text{crossing}}}
\text{sign}(\Gamma(\Psi,t))
+\sum_{\substack{a<t<b \\ t\ \text{crossing}}}
\text{sign}(\Gamma(\Psi,t)),$$ 
where sign denotes the signature of the quadratic form, i.e. the number of positive eigenvalues minus the number of negative eigenvalues.
\begin{proposition}(\cite{ROBBIN1993})
	The Robbin-Salamon index satisfies the following properties:
	\begin{enumerate}[label=(\roman*)]
		\item Let $a<b<c$, then
		
		$$\mu_{RS}(\Psi|_{[a,c]})=\mu_{RS}(\Psi|_{[a,b]})+\mu_{RS}(\Psi|_{[b,c]}).$$
		\item Suppose that $\Phi\in \Sigma(n)$ and that there exists a homotopy with fixed end points between $\Psi$ and $\Phi$, then
		
		$$\mu_{RS}(\Psi)=\mu_{RS}(\Phi).$$
		\item Suppose now that $\Phi\in \Sigma(m)$, then $\Psi\oplus\Phi \in \Sigma(n+m)$ and
		
		$$\mu_{RS}(\Psi\oplus\Phi)=\mu_{RS}(\Psi)+\mu_{RS}(\Phi).$$
	\end{enumerate}
	
\end{proposition}

We now recall the computation of the Robbin-Salamon index of Reeb orbits for the standard contact form on $S^3$. As with the Conley--Zehnder index, for a Reeb orbit $\gamma$, we define $\mu_{RS}(\gamma)$ to be $\mu_{RS}(\Psi)$, where $\Psi:[0,T]\to Sp(n)$ is the linearization of the Reeb flow $R_{0}$ restricted to the contact structure $\xi_0$ under the global trivialization, as defined above. Notice that we have $CZ(\gamma)-\mu_{RS}(\gamma)\in[0,1]$ for any orbit  $\gamma$. 

\begin{lemma}\label{lemma:nelson}(see \cite{Nelson2017AutomaticTI})
Let $\gamma$ be a simple closed orbit of $R_0$ in $S^3$. Then, for every $k>0$, we have
	$$\mu_{RS}(\gamma^k)=4k.$$
\end{lemma}

This lemma accounts for the contribution of the part ``$R_0$" of our flow, as seen in \eqref{eq: preqflow}. We now compute the contribution of the Hamiltonian flow in order to prove the following theorem.

\begin{theorem}\label{thm:RSdis}
	Let $f$ be a Morse--Bott function on $\mathbb{CP}^1$ and let $\gamma$ be the distinguished Reeb orbit of $(S^3,\lambda_f)$ associated to a nondegenerate critical point of $f$. If $\gamma^k$ is nondegenerate, then 
	\begin{equation}\label{eq:RSdis}
		CZ\left(\gamma^k\right)=4k-\left(\frac{1}{2}+\left\lfloor\frac{ k\sqrt{|\lambda_1^f\lambda_2^f|}}{2}\right\rfloor\right)\textup{sign}\left(\textup{Hess}_xf\right),
	\end{equation}
where $\lambda_i^f$ are the eigenvalues of $\op{Hess}_{x}f$ with respect to the Fubini-Study metric. 

If $\gamma^k$ is degenerate, then
$$\mu_{RS}\left(\gamma^k\right)=4k-\frac{ k\sqrt{|\lambda_1^f\lambda_2^f|}}{2}\textup{sign}\left(\textup{Hess}_xf\right).$$

 In particular, if $\gamma$ is a nondegenerate Reeb orbit projecting to a saddle point of $f$, then $CZ(\gamma^k)=4k$ and therefore it is positive hyperbolic.
\end{theorem}
\begin{proof}
Let $R_{f}$ be the Reeb vector field of $\lambda_f=e^{p^*f}\lambda_0$, let $\xi=\ker\lambda_f$ and fix $z\in p^{-1}(x)$. Since the Reeb vector field is transverse to the contact structure, there is a decomposition
	$$T_{(s,w)}(\bbR\times S^3)=\bbR\oplus \left<R_{\lambda_f}(z)\right>\oplus \xi_w.$$
Let $\varphi^t_f$ denote the flow of $R_0+\widetilde{X}_f$ under the global quaternionic trivialization. Then
	
	$$d\varphi^t_f(s,w)=\begin{pmatrix}
		\begin{pmatrix}
			1&0\\
			0&1
		\end{pmatrix}&\\
		&d\varphi^t_f|_{\xi_w}
	\end{pmatrix}$$
Now, let $\Phi_f(t)=d\varphi^t_f|_{\xi_z}$. Since $f$ is constant along the distinguished orbit, the linearized flow of $R_0+\widetilde{X}_f$ is simply a rescaling of the linearized flow of $R_f=e^{-p^*f}(R_0+\widetilde{X}_f)$. Therefore, we can use $\Phi_f(t)$ to compute the Robbin-Salamon index. Now, we denote by $\Phi_0(t)$ and $\widetilde{\Psi}_f(t)$ the linearized flows of $R_0$ and $\widetilde{X}_f$ restricted to the contact structure, respectively. We reparametrize the paths above so that they are all defined over the interval $[0,T]$, where $T=k\pi$ is the period of $\gamma^k$. 

Note that the flows $R_0$ and $\widetilde{X}_f$ commute because $\widetilde{X}_f$ was defined to be invariant under the Hopf flow. Therefore, the linearized flows also commute and $\Phi_f(t)=\Phi_0(t)\widetilde{\Psi}_f(t)$. Consequently,
$$\mu_{RS}(\Phi_f)=\mu_{RS}(\Phi_0\widetilde{\Psi}_f).$$
Since the product path $\Phi_0\widetilde{\Psi}_f|_{[0,T]}$ is homotopic to the concatenation of $\Phi_0|_{[0,T]}$ and $\widetilde{\Psi}_f|_{[0,T]}$, it follows that
	\begin{equation*}
		\mu_{RS}(\Phi_f)=\mu_{RS}(\Phi_0\widetilde{\Psi}_f)=\mu_{RS}(\Phi_0)+\mu_{RS}(\widetilde{\Psi}_f).
	\end{equation*}
By Lemma \ref{lemma:nelson}, the first summand is $4k$. As for the second summand, we need to understand the linearization of the flow of the horizontal lift $\widetilde{X}_f$ of $X_f$ in $S^3$. For this, let $dp:TS^3\xrightarrow{} TS^2 $ be the usual tangent bundle projection map, we can define a splitting of $T_zS^3=V_z\oplus H_z=\op{Span}(R_0)\oplus \xi_z$, where $V_z:=\ker dp_z$ is called the \textit{vertical subspace} and  $H_{(z,v)}$ the \textit{horizontal subspace}\footnote{this is a choice of an Ehresmann connection.}. Moreover, $dp$ induces a symplectomorphism between $(\xi_z,d
\lambda_0)$ and $(T_{\pi(z)}\mathbb{CP}^1,\omega)$. On the other hand, let $\psi_t$ be the flow of $X_f$ and $\widetilde{\psi}_t$ the flow for $\widetilde{X}_f$, then we have the following commuting diagram of symplectomorphisms.

\begin{center}
	\begin{tikzpicture}
		\matrix (m)
		[
		matrix of math nodes,
		row sep    = 3em,
		column sep = 4em,
		nodes={anchor=center}
		]
		{   			\xi_z          &  \xi_{\widetilde{\psi}_t(z)}\\
			T_{p(z)}\mathbb{CP}^1 &       T_{\psi_t(p(z))}\mathbb{CP}^1\\  };
		\path
		(m-1-1) edge [->] node [left] {$dp_z$} (m-2-1)
		
		(m-1-2)
		edge [->] node [right] {$dp_{\widetilde{\psi}_t(z)} $} (m-2-2)

		(m-1-1.east) edge [->] node [above] {$d \widetilde{\psi}_t|_{\xi_z}$} (m-1-2)
		
		(m-2-1)edge [->] node [below] {$d\psi_t$} (m-2-2);
		
	\end{tikzpicture}
\end{center}
Now, take $z\in S^3$ such that $p(z)=x$. We then need to relate the path of symplectic matrices $\widetilde{\Psi}_f(t)=\tau_0\circ (d\widetilde{\psi}_t|_{\xi}(x))\circ \tau_0^{-1}$ and the Hessian of $f$. In particular, as $\tau_0$ is a global trivialization, it induces a trivialization $\bar{\tau}$ of $(T_x \mathbb{CP}^1,\omega)$ such that $\bar{\tau} \circ dp=dp\circ \tau$ and the induced path $\Psi(t)=\bar{\tau}\circ (d\psi_t(x))\circ \bar{\tau}^{-1}$ is equal to $\tilde{\Psi}(t)$ up to a conjugation. Therefore, they have the same Robbin-Salamon index. In fact, the quaternionic trivialization induces a natural compatible almost complex structure $J$ for the symplectic vector bundle $(\xi,d\lambda_0)$, which induces a compatible almost complex  structure $\bar{J}$ in $(\mathbb{CP}^1,\omega)$. Moreover, Kähler coordinates are all equivalent in a neighborhood of the point $x$. Therefore, we can choose Kähler coordinates on $\mathbb{CP}^1$ around the point $x$ such that the linearized flow is given by the equation
$$\dot{\Psi}=-J_0\nabla^2 f\cdot \Psi,$$
on those coordinates, where $\nabla^2 f$ is the Hessian matrix of $f$ with respect to the Fubini-Study metric.  Using a standard calculation from \cite[Proposition 41]{gutt2012conleyzehnderindexpathsymplectic}, and observing that the nondegeneracy of $\gamma^k$
implies that $t=k\pi$ is not a crossing for $\Psi$, we conclude that the Robbin-Salamon index of the corresponding path of symplectic matrices is equal to
$$\mu_{RS}(\Psi)=-\left(\frac{1}{2}+\left\lfloor\dfrac{T\sqrt{|\lambda_1^f\lambda_2^f|}}{2\pi}\right\rfloor\right)\textup{sign}\left(\textup{Hess}_xf\right)=-\left(\frac{1}{2}+\left\lfloor\dfrac{k\sqrt{|\lambda_1^f\lambda_2^f|}}{2}\right\rfloor\right)\textup{sign}\left(\textup{Hess}_xf\right).$$
If $\lambda_1^f\lambda_2^f<0$, there is no crossing but $t=0$ and the same formula works.  
 
By the same calculation as above, if $\gamma^k$ is degenerate and $\lambda_1^f\lambda_2^f>0$, then $t=k\pi$ is a crossing for $\Psi$ and

$$\mu_{RS}(\Psi)=-\frac{ k\sqrt{|\lambda_1^f\lambda_2^f|}}{2}\textup{sign}\left(\textup{Hess}_xf\right).$$

\end{proof}

\subsection{Conley--Zehnder index of undistinguished orbits}

In order to compute the Conley--Zehnder index of undistinguished orbits, we will first construct toric coordinates. Let $f:\mathbb{CP}^1\to\mathbb{R}$ be a Morse--Bott function. Denote by $\Lambda_0$ the union of the connected components of $f^{-1}(\text{crit}(f))$ that contain a critical point and observe that $\bbC P^1\backslash \Lambda_0$ is a finite disjoint union of cylinders. Take any of those cylinders and denote it by $\mathcal{U}$. Under the identification
$$\mathcal{U}\cong \bbR/\pi\bbZ\times (y_0,y_1),$$
we can assume that $f(x,y)=y$. Let $S_y=(\bbR/\pi\bbZ)\times\{y\}\subset \mathcal{U}$ for $y\in (y_0,y_1)$, and for each $y$ we define $h(y)$ to be $\omega$-area of one of the connected component of $\bbC P^1\backslash S_y$. We first choose $h(y)$ to be an increasing function on $y$. We observe that $h(y)$ is the holonomy of the connection $\lambda_0$ over the circle $S_y$. Indeed, the holonomy $\op{Hol}_{\gamma}(\lambda_0)\in U(1)$ of a closed loop $\gamma:[0,1]\mapsto\mathbb{CP}^1$ is defined as the element that satisfies $\tilde{\gamma}(1)=\op{Hol}_{\gamma}(\lambda_0)\tilde{\gamma}(0)$, where $\tilde{\gamma}$ is the horizontal lift of $\gamma$. In our case, it is not difficult to show that the holonomy, under the identification $U(1)\cong\bbR/\pi\bbZ$, is given as the integral 
$$\op{Hol}_{\gamma}(\lambda_0)=\int_{\sigma}\lambda_0,$$
where $\sigma$ is the composition of the path $\widetilde{\gamma}$ taken in the opposite direction and the rotation along the Hopf fiber in the positive direction until the loop closes, which is well-defined modulo $\pi$. Notice that we can take a capping disk $\mathbb{D}$ of $\sigma$ and project to a capping disk of $S_y$ on $\bbC P^1$ with the right orientations. This disk is the same disk used in the definition of $h(y)$, we can then see that $\op{Hol}_{\gamma}(\lambda_0)=h(y)$. Furthermore, we have that $h'(y)$  is the period of the circle $S_y$, viewed as an orbit of the flow of $X_f$. Now, observe that by equation \eqref{eq: preqflow}, we have an $S^1$-family of periodic points on the level $S_y$ if, and only if,
\begin{equation}\label{eq: quotient}
	\frac{h'(y)+h(y)}{\pi}=\frac{a}{b},
\end{equation}
 for $a,b$ relatively prime positive integers, since we are assuming that both $h$ and $h'$ are positive. If we choose the complementary disk to define $h(y)$, we get that the holonomy is $\pi-h(y)$ and the period of $X_f$ is $-h'(y)$, so we sometimes may assume that $a$ can be negative.

Observe that the definition of $h$ still makes sense and is well defined for isolated extremal points and critical Morse--Bott circles of $f$. Moreover, $h$ is not well defined for $\Lambda\subset \Lambda_0$ be the union of the connected components of $f^{-1}(\text{crit}(f))$ containing a saddle point for $f$. Therefore, by taking a consistent choice of disks, we may consider $h$ as a function defined on $\bbC P^1\backslash \Lambda$. 
 
 In addition, we have the following consequence of the integrability of $S^1$-invariant domains.
 
 \begin{lemma}\label{lemmma:toric}
 	Let $f:\bbC P^1\mapsto\bbR$ be a Morse--Bott function 
 	\begin{enumerate}[label=(\roman*)]
 		\item If $f$ has only two isolated critical points, then the domain $\mathbb{X}_F\subset\mathbb{C}^2$ is symplectomorphic to a star-shaped toric domain.
 		\item Let $\Lambda$ be the union of the connected components of $f^{-1}(\text{crit}(f))$ containing a saddle point for $f$, then there exist a finite family of pieces of curves $\{\alpha_i\}_{i\in I}$ in $\bbR^2_{\geq 0}$ such that 
 		 $(S^3\backslash p^{-1}(\Lambda),\,\lambda_f)$ is strictly contactomorphic to $(\sqcup_{i\in I}\mu^{-1}(\text{Im}(\alpha_i)),\,\lambda_0)$
 	\end{enumerate}
 \end{lemma} 
\begin{proof}
Assume first that $f$ is a perfect Morse--Bott function, $h(y)$ is a well defined function on the whole $\bbC P^1$. Therefore, we can define action coordinates $(I_1,I_2)$ for regular points of $f$ as 
	\begin{align}\label{eq:aacord}
		I_1(y,r)=\int_{\sigma_1}r\lambda_f=re^{y}h(y)& &I_2(y,r)=\int_{\sigma_2}r\lambda_f=re^{y}(\pi- h(y)),
	\end{align}
	
where $\sigma_1$ is defined by following to flow of the horizontal lift $\widetilde{X}_f$ of $X_f$ until reaching the same Hopf fiber, namely $\sigma_0$, and then rotating in the positive direction along the Hopf fiber until closing the loop. We can take $\sigma_2$ in a similar way, first we follow $\sigma_0$ in the opposite direction, and then rotating positively in the Hopf fiber until closing the loop. Observe that $I_1+I_2=\pi r e^f$.

Now, consider $z\in S^3$ and $r\in [0,1]$, then we can define

\begin{align*}
	I_1(z,r)=e^{f(p(z))}rh(p(z))& &I_2(z,r)=e^{f(p(z))}r(\pi-h(p(z)))
\end{align*}

Observe that $I_1,I_2$ are now smooth functions. Indeed, it is straightforward to see that it is smooth for regular levels. For critical levels, we can use Morse--Bott charts to prove smoothness. Observe that $I_1,I_2$ define a new integrable system whose only singular levels are those associated to isolated critical points of $f$.

 In addition, we can recover the action coordinates $I_1,I_2$ in this new setting. Indeed, we can define $I_1$ as the integral over some curve that follows the Hamiltonian flow of either $\widetilde{X}_{I_1}$ or $-\widetilde{X}_{I_1}$  until reaching the same Hopf fiber, namely $\sigma$, then following the Hopf flow in the positive direction until closing the curve. Similarly, $I_2$ is redefined integrating over a curve that follows $\sigma$ in the opposite direction, then following the Hopf flow in the positive direction until closing the curve. Finally, the symplectomorphism can be extended to the isolated critical points by Eliasson's theorem in \cite{Eliasson1990NormalFF}.

	For the second part, we see that $I_1,I_2$ are not well defined and smooth globally. However, $I_1,I_2$ can be defined outside $p^{-1}(\Lambda)$ and are smooth. Hence, for each connected component of $S^3\backslash p^{-1}(\Lambda)$, we have an integrable system defined by $I_1,I_2$. Now, observe that $I_1,I_2$ are homogeneous and by \cite{Colombo_2025} we conclude that the action-angle coordinates actually induce a contactomorphism, i.e., if $(I_1,I_2,\phi_1,\phi_2)$ are the action-coordinates then
	
	\begin{equation}\label{eq:contactoric}
		I_1d\phi_1+I_2d\phi_2=\lambda_f
	\end{equation}
	
	for regular levels. For an isolated critical point $x$ of $f$, observe that although either $\phi_1$ or $\phi_2$ is not well defined. Assume without loss that $\phi_1$ is not well defined on $p^{-1}(x)$. However, $I_1d\phi_1$ is still well defined and equal to $0$. Therefore, we can extend this contactomorphism to the isolated critical points by continuity of \eqref{eq:contactoric}.
	
	Finally, we are going to denote $\alpha_i$ as a parametrization of the image of the $i$-th connected component under the map $(I_1,I_2)$. We conclude $(S^3\backslash p^{-1}(\Lambda),\lambda_f)$ and $(\sqcup_{i\in I}\mu^{-1}(\text{Im}(\alpha_i)),\lambda_0)$ are strictly contactomorphic.
	
\end{proof}


Now, we need to compute the index for a Reeb orbit in such $S^1$-families as a function of $a$ and $b$. In order to do so, let $S_y$ and $\mathcal{U}$ as before. Since the rotation number is defined locally around $S_y$, we can then modify $f$ outside of the cylinder $\mathcal{U}$ to a perfect Morse--Bott function $g:\mathbb{CP}^1\to \mathbb{R}$. Hence, Lemma \ref{lemmma:toric} reduces the calculation of the rotation number from $S^1$-invariant domains to toric domains.

 We briefly recall the relation between rotation number of orbits of a given star-shaped toric domain $X_\Omega$ and a parametrization $\alpha$ of $\overline{\partial_+\Omega}$. There are 2 types of Reeb orbits. The first type are orbits in the fiber of the points $\alpha(0)$ or $\alpha(1)$ in the coordinate axes. The second type of orbits come in $S^1$-families, in the fiber of a point $\alpha(t)$ such that
$$\frac{\alpha'_2(t)}{\alpha'_1(t)}\in \bbQ,$$
i.e. points where the slope of the curve is rational. Another way to describe it is by taking the $\nu=(\nu_1(t),\nu_2(t))$ the outward normal vector to the curve in a certain point $\alpha(t)$. Indeed, $\nu_2/\nu_1$ is rational if, and only if, the slope is rational, since  $-\frac{\nu_1}{\nu_2}=\frac{\alpha'_2(t)}{\alpha'_1(t)}$. We can then assume that $\nu$ has integer and relatively prime coordinates. Under this normalization, we have, by \cite{Gutt2022Examples}, that a Reeb orbit $\gamma$ that projects to $\alpha(t)$ has rotation number equal to 

\begin{itemize}
	\item $\rho(\gamma)=\nu_1+\nu_2$ if $t\in(0,1)$,
	\item $\rho(\gamma)=1-\op{slope}(\alpha)$ or  $\rho(\gamma)=1-(\op{slope}(\alpha))^{-1}$ if $t=0$ or $1$.
\end{itemize}
We can also extend this to pieces of curves. In this case, we have the same Reeb orbits we described. As the rotation number of orbits is local, i.e., the rotation number only depends on a neighborhood of an orbit $\gamma$, we have that the rotation numbers described in the same way.

 We can now use the toric description and the coordinates in \eqref{eq:aacord} to compute the rotation number and the Robbin-Salamon index of those undistinguished Reeb orbits. In toric coordinates,  we can see from \eqref{eq: quotient} that
\begin{equation}\label{eq:slope}
	\frac{\partial_yI_2}{\partial_yI_1}=\frac{\pi-h'(y)-h(y)}{h'(y)+h(y)}=\frac{b-a}{a}
\end{equation}
for regular levels of $f$. For critical levels of f, we have the slope is going to be the limit of \eqref{eq:slope}. If $a$ is positive, \eqref{eq:slope} is given by $-\frac{\nu_1}{\nu_2}$, where $\nu_2=a$ and $\nu_1=a-b$. In fact, the coordinates $(\nu_1,\nu_2)$ would be the integer outward normal vector on the curve $(I_1,I_2)$ and the rotation number of the orbits in the fiber would be equal to $\rho=\nu_1+\nu_2=2a-b$. 

On top of that, recall that the period of the flow of $\widetilde{X}_f$ at a regular level is given by $|h'(y)|$. Consequently, as approaching a Morse--Bott critical circle $|h'(y)|$ goes to $\infty$. By \eqref{eq:slope}, the slope on Morse--Bott critical circle is $-1$ and the rotation number is $2$.  

Putting all this together we have proved the following statement.

\begin{theorem}\label{thm: undior}
	Let $f:\bbC P^1\mapsto\bbR$ be a Morse--Bott function. Let $h$ be defined as before for all regular points. If 
	$$\frac{h'(y)+h(y)}{\pi}=\frac{a}{b}>0.$$
	Then the rotation number of the associated orbit is equal to $2a-b$. Moreover, the rotation number of the $S^1$-family of orbits associated to a Morse--Bott critical circle is $2$.
\end{theorem}

\subsection{Sufficient and necessary condition for dynamical convexity}
In this section, we use the previous results to give necessary and sufficient conditions for the domain $\mathbb{X}_F$ to be dynamically convex.

\begin{theorem}
	Let $f:\bbC P^1\mapsto\bbR$ be a Morse--Bott function. Let $h$ be defined as before for regular level sets $S_y$. Let $\lambda_1^f(x)$ and $\lambda_2^f(x)$ be the eigenvalues of the bilinear form induced by $\op{Hess}_xf$ over an isolated critical point $x\in \mathbb{CP}^1$ of $f$, with respect to the Fubini-Study metric. Then, $\mathbb{X}_F$ is dynamically convex if, and only if, 
	\begin{itemize}
		\item For every regular level set $S_y$ of $f$, $$\frac{h'(y)+h(y)}{\pi}\notin[0,1].$$
		\item Let $x$ be a local minimum to $f$. Then $\gamma_x$ is nondegenerate and 
		 $$|\lambda_1^f(x)\lambda_2^f(x)|< 4.$$
	\end{itemize}
	
\end{theorem}

\begin{proof}
	We first prove that this is a sufficient condition. Indeed, we proved in Theorem \ref{thm:RSdis} that the Conley--Zehnder index of distinguished orbits is given by \eqref{eq:RSdis}. Therefore, we can see that the condition $|\lambda_1^f(x)\lambda_2^f(x)|<4$ is sufficient and necessary  for the nondegenerate orbit associated to local minimum $x\in \bbC P^1$ to have Conley--Zehnder index greater than or equal to 3. Moreover, for a saddle point or a point of local maxima we have that $\op{sign}\op{Hess} f\leq 0$ , hence the Conley--Zehnder index is already greater than or equal to 3. In the case of undistinguished orbits, we proved in Theorem \ref{thm: undior} that the rotation number of an orbit in the regular level $y$ is equal to either $2a-b$ or $2b-a$, if $h'(y)+h(y)=\pi\frac{a}{b}$. Therefore, if $h'(y)+h(y)>\pi$, we have that $a>b>0$ and the rotation number on those orbits satisfies that $\rho=2a-b>1$. On the other hand, if $h'(y)+h(y)<0$ we know that if $\tilde{h}=\pi-h$, then $\tilde{h}$ satisfies that $\tilde{h}'(y)+\tilde{h}(y)>\pi$. Consequently, by Theorem \ref{thm: undior} the rotation number is greater than 1. At Morse--Bott critical level sets, the rotation number of the associated orbits is $2$, by Theorem \ref{thm: undior}.
	
	We now prove the converse statement. By Lemma \ref{lemmma:toric}, we can realize $(S^3\backslash p^{-1}(\Lambda),\lambda_f)$ as union of toric contact manifolds $\mu^{-1}(\text{Im}(\alpha_i))$, where $\alpha_i$ are pieces of curves in $\bbR_{\geq 0}$. Similar to the proof of Theorem \ref{thm: undior}, we can use this description to describe all orbits by the toric picture, except for orbits that come from saddle critical points of $f$. The orbits coming from saddle points have Conley--Zehnder equal to 4, by Theorem \ref{thm:RSdis}.

	We have already analyzed the orbits related to critical points and we are left to analyze undistinguished Reeb orbits related to the pieces $\mu^{-1}(\text{Im}(\alpha_i))$. It follows from the proof of Theorem \ref{thm: undior} that the slope is $\frac{b-a}{a}$, this is negative if and only if $a>b$ or $a<0$, which in turn is equivalent to $h'(y)+h(y)\notin [0,\pi]$. Therefore, we need to prove that the slope at every point of the curve $\alpha_i$, is negative. Indeed, take $\alpha_i$ to be one of those pieces of curve. Suppose without loss of generality that $\alpha_i$ intersects a coordinate axis. For this intersection, we have an associated orbit $\gamma$ such that $\rho(\gamma)=1-\op{slope}(\alpha)$ or  $\rho(\gamma)=1-(\op{slope}(\alpha))^{-1}$. Consequently, the slope on this point is negative. Within the same piece of curve $\alpha_i$, observe that every point has negative slope. In fact, if there was a point of positive slope, there would exist a point of vertical or horizontal slope, which corresponds to an orbit of rotation number $1$.	Suppose now that $\alpha_i$ has no intersection with the coordinate axis. Recall that the period of $X_f$ in a regular level $S_y$ is given by $|h'(y)|$. As we approach a saddle point of $f$, $|h'(y)|$ goes to infinity. By \eqref{eq: quotient} and \eqref{eq:slope}, the slope of a regular level set close to a saddle point would be close to $-1$. Notice that $\mu^{-1}(\text{Im}(\alpha_i))$ always has points that correspond to a point in $S^3$ arbitrarily close to a saddle point. Therefore, $\alpha_i$ always has a point with negative slope, and therefore every point has negative slope by the same argument as before.

\end{proof}

 \begin{proof}[Proof of Theorem \ref{thm:characdynconv}]
 	We can ignore the saddle points since we know from Theorem \ref{thm:RSdis} that the Conley--Zehnder index of its corresponding orbit is 4. Moreover, the condition $\frac{h'(y)+h(y)}{\pi}\notin[0,1]$ is equivalent to the corresponding curves being monotone up to the axis points. For the axis points, as the Conley--Zehnder is a local notion, having Conley--Zehnder greater than or equal to 3 in the axis orbits is equivalent to being monotone at that point.
 \end{proof}

\bibliographystyle{plain}
\bibliography{ref}

\end{document}